\documentclass[11pt, a4paper]{article}
\usepackage{}
\usepackage{amsthm}
\usepackage{mathrsfs}
\usepackage{amsmath,amssymb,latexsym,color}
\usepackage[colorlinks,
linkcolor=blue,
anchorcolor=green,
citecolor=magenta
]{hyperref}
\usepackage{graphicx}
\usepackage{tikz}
\usetikzlibrary{calc}
\usepackage{CJK}

\usepackage{authblk}

\long\def\delete#1{}

\usepackage{color}

\definecolor{Blue}{rgb}{0,0,1}
\definecolor{Red}{rgb}{1,0,0}
\definecolor{DarkGreen}{rgb}{0,0.6,0}
\definecolor{DarkYellow}{rgb}{1,1,0.2}
\definecolor{DarkPurple}{rgb}{.6,0,1}

\usepackage{xcolor}
\usepackage[normalem]{ulem}

\usepackage{cleveref}
\crefformat{section}{\S#2#1#3}
\crefformat{subsection}{\S#2#1#3}
\crefformat{subsubsection}{\S#2#1#3}
\crefrangeformat{section}{\S\S#3#1#4 to~#5#2#6}
\crefmultiformat{section}{\S\S#2#1#3}{ and~#2#1#3}{, #2#1#3}{ and~#2#1#3}

\def\ma{\mathscr{A}}
\def\mb{\mathscr{B}}
\def\mc{\mathscr{C}}

\def\me{\mathscr{E}}
\def\mf{\mathscr{F}}

\def\mh{\mathscr{H}}

\def\ml{\mathscr{L}}
\def\mm{\mathscr{M}}
\def\mp{\mathscr{P}}

\def\ms{\mathscr{S}}

\def\mw{\mathscr{W}}

\def\my{\mathscr{Y}}

\def\bs{\setminus}

\def\ff{\mathbb{F}_q}

\def\ge{\geqslant}
\def\le{\leqslant}
\def\b{\brack}

\def\kn{\mp_m}
\def\ro{\romannumeral}

\newtheorem{thm}{Thoerem}[section]
\newtheorem{lem}[thm]{Lemma}

\newtheorem{pr1}[thm]{Proposition}

\begin{document}
	\setcounter{page}{1}
	\renewcommand{\thefootnote}{}
	\newcommand{\remark}{\vspace{2ex}\noindent{\bf Remark.\quad}}
	\renewcommand{\abovewithdelims}[2]{%
		\genfrac{[}{]}{0pt}{}{#1}{#2}}

	
	\def\qed{\hfill$\Box$\vspace{11pt}}
	
	\title {\bf  Non-trivial $t$-intersecting families for symplectic polar spaces}

	\author{Tian Yao\thanks{E-mail: \texttt{yaotian@mail.bnu.edu.cn}}}
	\author{Benjian Lv\thanks{E-mail: \texttt{bjlv@bnu.edu.cn}}}
	\author{Kaishun Wang\thanks{Corresponding author. E-mail: \texttt{wangks@bnu.edu.cn}}}
	\affil{\small Sch. Math. Sci. {\rm \&} Lab. Math. Com. Sys.,
		Beijing Normal University, Beijing, 100875,  China}

	\date{}
	
	\openup 0.5\jot
	\maketitle

	\begin{abstract}
		
		Let $\mp$ be a symplectic polar space over a finite field $\ff$, and $\mp_m$ denote the set of all $m$-dimensional subspaces in $\mp$. We say a $t$-intersecting subfamily of $\mp_m$ is trivial if there exists a $t$-dimensional subspace contained in each member of this family. In this paper, we determine the structure of maximum sized non-trivial $t$-intersecting subfamilies of $\mp_m$.
		
		\vspace{2mm}
		
		
		\
		
		
		
	\end{abstract}
	
	\section{Introduction}
	
	Let $n$ and $k$ be two positive integers with $n\ge k$.
	For an $n$-element set $X$, denote the set of all $k$-subsets of $X$ by $\binom{X}{k}$.
	Given a positive integer $t$, we say a family $\mf\subset\binom{X}{k}$ is \emph{$t$-intersecting} if $|A\cap B|\ge t$ for any $A,B\in\mf$.
	A $t$-intersecting family is called \emph{trivial} if every element of this family contains a fixed $t$-subset of $X$.
	The famous Erd\H{o}s-Ko-Rado theorem \cite{EKR,Fn,Wn} states that a $t$-intersecting family with maximum size is trivial when $n>(t+1)(k-t+1)$.
	The structure of non-trivial $t$-intersecting families with maximum size was completely determined in 1996 \cite{SHMT,SHM}. Recently, some other maximal non-trivial $t$-intersecting families were described in \cite{LSET2,LSET3,LSET1}.
	

	Let $V$ be an $n$-dimensional vector space over the finite field $\ff$, where $q$ is a prime power, and ${V\b k}_q$ denote the family of all $k$-dimensional subspaces of $V$.
	We usually replace ``$k$-dimensional subspace" with ``$k$-subspace" for short.
	Let $\mf\subset{V\b k}_q$ be a family with $\dim(F_1\cap F_2)\ge t$ for any $F_1,F_2\in\mf$, which we say is \emph{$t$-intersecting}. In particular, $\mf$ is called \emph{trivial} if there exists a $t$-subspace contained in each element of $\mf$.
	The Erd\H{o}s-Ko-Rado Theorem for vector space \cite{VEKR2,VEKR1,VEKR3} shows that a $t$-intersecting subfamily of ${V\b k}_q$ with maximum size is trivial when $\dim V>2k$.
	The structure of maximum sized non-trivial $t$-intersecting subfamily of ${V\b k}_q$ was determined via the parameter ``$t$-covering number" \cite{VHM1,VHM2}.

	Let $f$ be a non-degenerate sesquilinear form defined on $V$. An $m$-subspace $M$ of $V$ is called \emph{totally isotropic} if $f(x,y)=0$ holds for any $x,y\in M$. Denote the set of all totally isotropic subspaces of $V$ with respect to $f$ by $\mp$. Equipped with the inclusion relation, $\mp$ is a \emph{classical polar space}, denoted by the same symbol $\mp$. There are six kinds of classical polar spaces \cite{CPS}. In this paper, we consider the case that $\mp$ is a \emph{symplectic polar space}, where $n=2\nu$ and $f$ is an alternating bilinear form.
	
	For $0\le m\le\nu$, denote the set of all $m$-subspaces in $\mp$ by $\mp_m$. Let $\mf$ be a subfamily of $\mp_m$. If $\dim(F_1\cap F_2)\ge t$ for any $F_1, F_2\in\mf$, we say $\mf$ is $t$-\emph{intersecting}.
	If there exists a fixed $t$-subspace contained in each member of $\mf$, we say $\mf$ is \emph{trivial}.
	The maximum sized $t$-intersecting subfamily of $\mp_m$ was proven to be trivial if $\nu\ge m$ for $t=1$ \cite{SPEKR2,SPEKR1}, or $2\nu\ge3m+1$ with $(2\nu,q)\neq(3m+1,2)$ for $t\ge2$ \cite{SPEKRT}.

	For a subspace $A$ of $V$, let $\mm(m;A)$ denote the set of all members of $\mp_m$ contained in $A$. For $T\in\mp_t$ and $U\in\mp_m$ with $\dim(T\cap U)=m-1$, write
	$$\mh_1(T,U)=\{F\in\mp_m: T\subset F, \dim(F\cap U)\ge t\}\cup\mm(m;T+U).$$
	For $Z\in{V\b t+2}$, write
	$$\mh_2(Z)=\{F\in\mp_m: \dim(F\cap Z)\ge t+1\}.$$

	Our main result states the structure of non-trivial $t$-intersecting subfamilies of $\mp_m$ with maximum size under certain conditions.

	\begin{thm}\label{SSHMT}
		Let $\nu,m$ and $t$ be integers such that $\nu\ge m+1=t+2$, or $2\nu\ge\max\{3m+2,2m+2t+6\}$ with $m\ge t+2$.
		Suppose $\mf$ is a maximum sized non-trivial $t$-intersecting subfamily of $\mp_m$.
		Then one of the following holds.
		\begin{itemize}
			\item[\rm{(\ro1)}] If $1\le t\le\frac{m}{2}-1$, then $\mf=\mh_1(T,U)$ for some $T\in\mp_t$ and $U\in\mp_m$ with $T+U\in\mp_{m+1}$.
			\item[\rm{(\ro2)}] If $\frac{m}{2}-1<t\le m-1$, then $\mf=\mh_2(Z)$ for some $Z\in\mp_{t+2}$,
			or $\mf=\mh_1(T,U)$ for some $T\in\mp_1$ and $U\in\mp_3$ with $T+U\in\mp_{4}$ if $(m,t)=(3,1)$.
		\end{itemize}
	\end{thm}

	\section{Inequalities for sizes of two $t$-intersecting families}

	In this section, we always auusme that $m\ge t+2$. We prove bounds for the families $\mh_1(T,U)$ for $T\in\mp_t$ and $U\in\mp$ with $\dim(T\cap U)=t-1$. Then we compare the size of $\mh_1(T,U)$ with the size of $\mh_2(Z)$ under the condition that $T+U\in\mp_{m+1}$ and $Z\in\mp_{t+2}$.

	For positive integers $a$ and $b$ with $a\ge b$, define the \emph{Gaussian binomial coefficient} by
	$${a\b b}_q:=\prod_{0\le i<b}\dfrac{q^{a-i}-1}{q^{b-i}-1},$$
	and set ${a\b0}_q=1$.
	It is routine to check that the size of ${V\b k}_q$ is ${n\b k}_q$.
	From now on, we will omit the subscript $q$. We present some properties for the Gaussian binomial coefficient and the symplectic polar space, which will be used throughout our proof.
	
	For positive integers $m,i$ with $m>i$, we have
	\begin{equation}\label{FFS}
		q^{m-i}<\frac{q^m-1}{q^i-1}<q^{m-i+1},\quad q^{i(m-i)}\le{m\b i}<q^{i(m-i+1)}.
	\end{equation}
	For positives $m_1,m$ with $m_1\le m\le\nu$, let $N(\nu,m,m_1)$ be the number of members of $\mp_{m}$ containing a fixed member of $\mp_{m_1}$. From \cite[Theorem 3.38]{JS}, we derive that
	\begin{equation}\label{JS}
		N(\nu,m,m_1)=\prod\limits_{i=1}^{m-m_1}\dfrac{q^{2(\nu-m+i)}-1}{q^{i}-1}.
	\end{equation}
	For $1\le m\le \nu-1$, let $Q$ be an $m$-subspace of $V$ and $\alpha_1,\dots,\alpha_m$ a basis of $Q$. Note that the rank of matrix $(f(\alpha_i,\alpha_j))_{m\times m}$ does not depend on the choose of the basis. It is routine to check that its rank is $2s$ for some non-negative integer $s$. Then we say $Q$ has \emph{type $(m,s)$}. Note that $Q$ has type $(m,0)$ if and only if $Q$ is totally isotropic. By \cite[Theorem 3.27]{JS}, the number of members of $\mp_m$ contained in a fixed $(m+1,1)$-type subspace is $q+1$.

	For $\mf\subset\mp_m$ and a subspace $A$ of $V$, let $\mf_A$ denote the set of members of $\mf$ containing $A$.  Write
	$$f_0(\nu,m,t)={m-t+1\b1}N(\nu,m,t+1)-q{m-t+1\b2}N(\nu,m,t+2).$$

	\begin{lem}\label{PC}
		Let $\nu,m,t$ be positive integers with $m\ge t+2$ and $2\nu\ge3m-t+3$.
		Suppose $T\in\mp_t$ and $U\in\mp_m$ with $\dim(T\cap U)=t-1$.  Then the following hold.
		\begin{itemize}
			\item[\rm{(\ro1)}] If $T+U\in\mp_{m+1}$, then $|\mh_1(T,U)|>f_0(\nu,m,t)$.
			\item[\rm{(\ro2)}] If $T+U\not\in\mp_{m+1}$, then $|\mh_1(T,U)|<f_0(\nu,m,t)$.
		\end{itemize}
	\end{lem}
	\begin{proof}
		Observe that $\dim(T+U)=m+1$.

		(\ro1)	Write all $(t+1)$-subspaces of $T+U$ containing $T$ as $W_1,\dots,W_{{m-t+1\b1}}$.
		For each $i=1,\dots,{m-t+1\b1}$, let $\mw_i$ denote the set of all $m$-subspaces in $\mp$ containing $W_i$.
		For $F\in\mp_m$ with $T\subset F$, if $\dim(F\cap(T+U))=l\ge t+1$, it appears ${l-t\b1}$ times in $\mw_1,\dots,\mw_{{m-t+1\b 1}}$. For each $i=1,2,\dots,m-t$, write $\me_i=\{F\in\mp_m: T\subset F,\ \dim(F\cap(T+U))=t+i\}$.
		Then we have
		\begin{equation*}
			\begin{aligned}
				{m-t+1\b1}N(\nu,m,t+1)&=\sum_{i=1}^{m-t+1\b1}|\mw_i|=\sum_{i=1}^{m-t}{i\b 1}|\me_i|.	
			\end{aligned}
		\end{equation*}
		Similarly, we have
		$$q{m-t+1\b2}N(\nu,m,t+2)=\sum_{i=2}^{m-t}q{i\b 2}|\me_i|.$$
		From ${i\b1}<q{i\b2}$ for $i\ge3$, we obtain
		\begin{equation}\label{PC2}
			\begin{aligned}
				f_0(\nu,m,t)\le\left|\left\{F\in\mp_m: T\subset F,\ t+1\le\dim(F\cap(T+U))\le t+2\right\}\right|.
			\end{aligned}
		\end{equation}
		Then (\ro1) follows from (\ref{PC2}) and the definition of $\mh_1(T,U)$.

		(\ro2) Since $T\in\mp_t$, $U\in\mp_m$ and $T+U\not\in\mp_{m+1}$, there exist $\alpha\in T\bs U$ and $\gamma\in U\bs T$ such that $f(\alpha,\gamma)\neq0$.
		Therefore,
		$$|\{S\in\mm(t+1;T+U): T\subset S\}|\le\left|\left\{S\in{T+U\b t+1}: T\subset S\right\}\bs\{\langle T,\gamma\rangle\}\right|={m-t+1\b1}-1.$$
		Notice that
		$$\{F\in\mp_m: T\subset F, \dim(F\cap(T+U))\ge t+1\}\subset\bigcup_{S\in\mm(t+1;T+U),\ T\subset S}\left(\mh_1(T,U)\right)_S$$
		and $T+U$ has type $(m+1,1)$. Then we have
		$$|\mh_1(T,U)|\le\left({m-t+1\b1}-1\right)N(\nu,m,t+1)+(q+1).$$
		By (\ref{FFS}), we get
		\begin{equation*}
			\begin{aligned}
				\dfrac{|\mh_1(T,U)|-f_0(\nu,m,t)}{N(\nu,m,t+2)}&\le-\dfrac{q^{2(\nu-t-1)}-1}{q^{m-t-1}-1}+q{m-t+1\b2}+\dfrac{q+1}{N(\nu,m,t+2)}\\
				&<-q^{2\nu-m-t-1}+q^{2(m-t)+1}+q+1\\
				&<0,
			\end{aligned}
		\end{equation*}
		as desired.
	\end{proof}

	Let $T\in\mp_t$, $U\in\mp_m$ with $T+U\in\mp_{m+1}$, and $Z\in\mp_{t+2}$. From \cite[Lemma 3.11]{JS}, we derive that $|\mh_1(T,U)|$ and $|\mh_2(Z)|$ do not depend on the choice of $T$, $U$ and $Z$. Write
	$$h_1(\nu,m,t)=|\mh_1(T,U)|,\quad h_2(\nu,m,t)=|\mh_2(Z)|.$$
	Notice that
	$$h_2(\nu,m,t)={t+2\b1}N(\nu,m,t+1)-q{t+1\b1}N(\nu,m,t+2).$$

	\begin{lem}\label{C2}
		Let $\nu,m$ and $t$ be positive integers with $m\ge t+2$ and $2\nu\ge 3m+1$. Then the following hold.
		\begin{itemize}
			\item[\rm{(\ro1)}] If $1\le t\le\dfrac{m}{2}-1$, then $h_2(\nu,m,t)<h_1(\nu,m,t)$.
			\item[\rm({\ro2)}] If $\dfrac{m}{2}-1< t\le m-2$, then $h_2(\nu,m,t)\ge h_1(\nu,m,t)$ and equality holds if and only if $(m,t)=(3,1)$.
		\end{itemize}
	\end{lem}

	\begin{proof}

		(\ro1) By (\ref{FFS}), (\ref{JS}) and Lemma \ref{PC}, we have
		\begin{equation*}
			\begin{aligned}
				&\dfrac{h_1(\nu,m,t)-h_2(\nu,m,t)}{N(\nu,m,t+2)}\\
				>&\left({m-t+1\b1}-{t+2\b1}\right)\dfrac{q^{2(\nu-t-1)}-1}{q^{m-t-1}-1}-q\left({m-t+1\b2}-{t+1\b1}\right) \\
				>& q^{t+2}\cdot q^{m-2t-2}\cdot q^{2\nu-t-m-1}-q^{2(m-t)+1}\\
				>&0.
			\end{aligned}
		\end{equation*}
		Hence $h_1(\nu,m,t)>h_2(\nu,m,t)$.

		(\ro2) Let $U\in\mp_m$, $T\in\mp_t$ with $T+U\in\mp_{m+1}$. Suppose $\dfrac{m}{2}-1< t\le m-3$ and $\mf=\mh_1(T,U)$.
		Write
		$$\ma_j=\left\{F\in\mp_m: T\subset F,\ \dim(F\cap(T+U))=j\right\}$$
		for each $t+1\le j\le m$, and
		$$\ml=\left\{(I,F)\in\mp_{t+1}\times\mp_m: T\subset I\subset T+U,\ I\subset F\right\}.$$
		Let $\mb=\bigcup\limits_{j=t+1}^m\ma_j$.
		We have
		$$|\mf\bs\mb|=|\{F\in\mm(m;T+U): T\not\subset F\}|=q^{m-t+1}{t\b1}.$$
		By double counting $|\ml|$, we have
		\begin{equation*}\label{BT1}
			{m-t+1\b1}N(\nu,m,t+1)=|\ml|=\sum_{j=t+1}^m|\ma_j|\cdot{j-t\b1}=|\mb|+\sum_{j=t+2}^mq{j-t-1\b1}|\ma_j|.
		\end{equation*}
		Since $m\ge t+3$, by \cite[Theorem 2.10]{QY}, we have
		\begin{equation*}
			\begin{aligned}
				|\ma_{m-1}|={m-t+1\b2}\left((q^{\nu-m+1}+q^2){\nu-m-1\b1}+q^{2(\nu-m)}{2\b1}\right).
			\end{aligned}
		\end{equation*}
		Then by (\ref{FFS}), we get
		$$\dfrac{|\ma_{m-1}|}{|\mf\bs\mb|}>\dfrac{q^{2(m-t-1)}\cdot q^{2(\nu-m)-1}}{q^{m-t+1}\cdot q^t}\ge q^{3m-2t-2}\cdot q^{-(m+1)}>1.$$
		Observe that $\mf=(\mf\bs\mb)\cup\mb$. It is routine to check that
		\begin{equation*}
			\begin{aligned}
				h_1(\nu,m,t)=|\mf|=|\mb|+|\mf\bs\mb|<|\ml|-(q-1)|\ma_{t+2}|.
			\end{aligned}
		\end{equation*}
		If $2t>m-1$, since $2\nu\ge3m+1$, by (\ref{FFS}), (\ref{JS}) and Lemma \ref{PC}, we obtain
		\begin{equation*}
			\begin{aligned}
				\dfrac{h_2(\nu,m,t)-h_1(\nu,m,t)}{N(\nu,m,t+2)}&>\left({t+2\b1}-{m-t+1\b1}\right)\dfrac{q^{2(\nu-t-1)}-1}{q^{m-t-1}-1}-q{t+1\b1}\\
				&\ge q^{m-t+1}\cdot q^{2t-m}\cdot q^{2\nu-m-t-1}-q^{t+2}\\
				&>0.
			\end{aligned}
		\end{equation*}
		If $2t=m-1$, from (\ref{FFS}), (\ref{JS}), $2\nu\ge3m+2$, $t\ge2$ and \cite[Theorem 2.10]{QY}, we obtain
		\begin{equation*}
			\begin{aligned}
			&h_2(\nu,m,t)-h_1(\nu,m,t)\\>&(q-1)|\ma_{t+2}|-q{t+1\b1}N(\nu,m,t+2)\\
			>&(q-1){t+2\b 2}\cdot q^{\frac{t(t-1)}{2}+2(t-1)(\nu-2t-2)+2(t-1)}-q^t{t+1\b1}q^{2(t-1)(\nu-2t-2)+\frac{t(t-1)}{2}+2(t-1)}\\
			>&0.
			\end{aligned}
		\end{equation*}
		Hence $h_1(\nu,m,t)<h_2(\nu,m,t)$.

		Now suppose $m=t+2$. Let $M\in\mm(t+2;T+U)$ containing $T$.
		Write
		\begin{equation*}
			\begin{aligned}
				\mb_0&=\{F\in\mm(t+2;T+U): T\not\subset F\},\\
				\mb_1&=\{F\in\mp_{t+2}: T\not\subset F,\ \dim(F\cap M)=t+1\},\\
				\mb_2&=\{F\in\mp_{t+2}: T\subset F,\ \dim(F\cap M)\ge t+1\},\\
				\mb_3&=\{F\in\mp_{t+2}: F\cap M=T,\ \dim(F\cap(T+U))=t+1\}.
			\end{aligned}
		\end{equation*}
		Observe that
		$|\mb_0|={t+3\b1}-{3\b1}=q^3{t\b1}$.
		By \cite[Theorem 2.10]{QY}, we have
		$$
		|\mb_1|=\left((q^{\nu-t+1}+q^3){\nu-t-2\b1}+q^{2(\nu-t)-1}\right){t\b1},
		$$
		$$
		|\mb_3|=(q^{\nu-t+1}+q^4){\nu-t-3\b1}+q^{2(\nu-t-1)}{2\b1}.
		$$
		Notice that $\mh_1(T,U)=\mb_0\overset{.}{\cup}\mb_2\overset{.}{\cup}\mb_3$, $\mh_2(M)=\mb_1\overset{.}{\cup}\mb_2$, which imply that $h_2(\nu,m,t)-h_1(\nu,m,t)=|\mb_1|-|\mb_0|-|\mb_3|$.
		When $t=1$, it is routine to chech that $h_2(\nu,3,1)-h_1(\nu,3,1)=0$.
		If $t\ge2$, from $2\nu\ge3m+1=3t+7$ and (\ref{FFS}), we have
		\begin{equation*}
			\begin{aligned}
				&h_2(\nu,t+2,t)-h_1(\nu,t+2,t)\\
				\ge&\left((q^{\nu-t+1}+q^3)q^{\nu-t-3}+q^{2(\nu-t)-1}\right)q^{t-1}-\left((q^{\nu-t+1}+q^4)q^{\nu-t-3}+q^{2(\nu-t)}\right)-q^{t+3}\\
				\ge&q^{2\nu-2t-2}-q^{t+3}\\
				>&0,
			\end{aligned}
		\end{equation*}
		as desired.
	\end{proof}

	\section{Proof of Theorem \ref{SSHMT}}
	
	Suppose $\mf\subset\mp_m$ is a maximum sized non-trivial $t$-intersecting family. When $m=t+1$ with $\nu\ge m+1$, by \cite[Remarks (\ro2) in Section 9.3]{T1}, there exists $Z\in{V\b t+1}$ such that $\mf\subset{Z\b t+1}$. Observe that $Z\in\mp_{t+2}$ or $Z$ has type $(t+2,1)$. Then it is routine to check that $\mf=\mh_2(Z)$ where $Z\in\mp_{t+2}$. To finish the proof of Theorem \ref{SSHMT}, in view of Lemmas \ref{PC} and \ref{C2}, it is sufficient to show that $|\mf|<f_0(\nu,m,t)$ if $m\ge t+2$ with $2\nu\ge\max\{3m+2,2m+2t+6\}$ and $\mf$ is not one of the expected families in Theorem \ref{SSHMT}. We first prove the following lemma.

	\begin{lem}\label{FS}
		Let $A\in\mm(a,0;2\nu)$ with $a\le m-1$, and $\my\subset\mp_m$ be a $t$-intersecting family. Suppose that there exists $Y\in\my$ such that $\dim(A\cap Y)=r\le t-1$. Then there exists $R\in\mp_{a+t-r}$ with $A\subset R$ such that
		$$|\my_A|\le{m-r\b t-r}|\my_R|.$$
	\end{lem}
	\begin{proof}
		W.l.o.g., assume that $\my_A\neq\emptyset$. Write
		$$\mh=\{H\in\mm(a+t-r,0;A+Y): A\subset H\}.$$
		Since $\my$ is $t$-intersecting, for each $G\in\my_A$, we have $\dim(G\cap Y)\ge t$, which implies that $\dim(G\cap(A+Y))\ge a+t-r$. Consequently, $\mh\neq\emptyset$.
		Then there exists $R\in\mh$ such that $|\my_B|\le|\my_R|$ for any $B\in\mh$.
		Observe that
		$$|\mh|\le{m-r\b t-r},\quad\my_A\subset\bigcup_{B\in\mh}\my_B.$$
		Then the desired result holds.
	\end{proof}

	For $S\in\mp_s$, we say $S$ is a $t$-\emph{cover} of $\mf$ if $\dim(S\cap F)\ge t$ for any $F\in\mf$.
	Write
	$$\tau_t(\mf):=\min\{\dim S: S\ \text{is a}\ t\text{-cover of}\ \mf\}.$$
	Since $\mf$ is non-trivial and each member of $\mf$ is a $t$-cover of $\mf$, we have $t+1\le\tau_t(\mf)\le m$. Notice that each member of $\kn$ containing $S$ is also a member of $\mf$ because of the maximality of $\mf$. The proof in the following is divided into two cases.

	\subsection{The case $\tau_t(\mf)=t+1$}

	\noindent\textbf{Assumption 1.} Let $\nu,m,t$ be positive integers with $2\nu\ge\max\{3m+2,2m+2t+6\}$ and $m\ge t+2$. Suppose $\mf$ is a maximum sized non-trivial $t$-intersecting subfamily of $\mp_m$ with $\tau_t(\mf)=t+1$ and $\ms$ is the set of all $t$-cover of $\mf$ with dimension $t+1$. Set $X=\sum\limits_{S\in\ms}S$.

	\begin{lem}\label{XJ}
		Let $\nu,m,t,\mf,\ms$ and $X$ be as in Assumption 1. Then $\ms$ is a $t$-intersecting family with $t\le\tau_t(\ms)\le t+1$. Moreover, if $\tau_t(\ms)=t$, then $t+1\le\dim X\le m+1$ and $\dim(F\cap X)=\dim X-1$ for any $F\in\mf\bs\mf_T$, where $T\in\mp_t$ is contained in each member of $\ms$.
	\end{lem}
	\begin{proof}
		Suppose for contradiction that there exist $S_1,S_2\in\ms$ such that $\dim(S_1\cap S_2)<t$.
		By (\ref{JS}), the number of $A\in\mp_{t+2}$ containing $S_2$ is $N(\nu,t+2,t+1)={2(\nu-t-1)\b1}$. Observe that the number of members of $\mm(t+2;S_1+S_2)$ containing $S_2$ is at most ${t+1\b1}$. Since $2\nu>3(t+1)$, we have ${2(\nu-t-1)\b1}>{t+1\b1}$. Then
		\begin{equation*}
			\begin{aligned}
				\{A\in\mp_{t+2}: S_2\subset A\not\subset S_1+S_2\}=&\{A\in\mp_{t+2}: S_2\subset A\}\bs\{A\in\mp_{t+2}: A\subset S_1+S_2\}\neq\emptyset.
			\end{aligned}
		\end{equation*}
		For $S_2'\in\mp_{t+2}$ with $S_2\subset S_2'\not\subset S_1+S_2$, it is routine to check that $S_2'\cap S_1=S_2\cap S_1$. Repeating this discussion several times, we obtain $U_1, U_2\in\mp_m$ containing $S_1$ and $S_2$, respectively, satisfying that $U_1\cap U_2=S_1\cap S_2$. By the maximality of $\mf$, we have $U_1,U_2\in\mf$, a contradiction to the fact that $\mf$ is $t$-intersecting, as desired. From $\ms$ is $t$-intersecting, we get $t\le\tau_t(\ms)\le t+1$.
		
		Suppose $\tau_t(\ms)=t$ and $T$ is a member of $\mp_t$ contained in each member of $\ms$. 	For each $F\in\mf\bs\mf_T$ and $S\in\ms$, observe that $\dim(S\cap F)=t$. Then
		$$m+1\le\dim(T+F)\le\dim(S+F)=m+1.$$
		Note that $T+F\subset S+F$. We get $S+F=T+F$, which implies that $T+F=X+F$. From $\dim(T+F)=m+1$, we get $\dim(F\cap X)=\dim X-1$ and $\dim X\le m+1$. Together with $\dim X\ge \dim S=t+1$, the desired result holds.
	\end{proof}

	\begin{pr1}
		Let $\nu,m,t,\mf,\ms$ and $X$ be as in Assumption 1. If $\tau_t(\ms)=t$ and $\dim X=m+1$, then $\mf=\mh_1(T,U)$ for some $T\in\mp_t$ and $U\in\mp_m$ with $T+U\in\mp_{m+1}$.
	\end{pr1}

	\begin{proof}
		Suppose $T\in\mp_t$ is contained in each member of $\ms$.  For each $U\in\mf\bs\mf_T$, by Lemma \ref{XJ}, we have $\dim(U\cap X)=m$ and $X=T+U$, which imply that
		$$\mf\bs\mf_T\subset\left\{F\in\mm(m;T+U): T\not\subset F\right\}.$$
		For each $G\in\mf_T$, since $\dim(U\cap G)\ge t$ and $\dim(T\cap U)=t-1$, we have $\dim(G\cap(T+U))\ge t+1$, which implies that
		$$\mf_T\subset\left\{F\in\mp_m: T\subset F,\ \dim(F\cap(T+U))\ge t+1\right\}.$$
		Thus $\mf\subset\mh_1(T,U)$. Observe that $\mh_1(T,H)$ is non-trivially $t$-intersecting. By the maximality of $\mf$, we have $\mf=\mh_1(T,U)$.
		
		If $X\not\in\mp_{m+1}$, by Lemma \ref{PC}, we have $|\mf|<h_1(\nu,m,t)$. Note $\mf$ is maximum sized. Then we get  $T+U\in\mp_{m+1}$.
	\end{proof}

	\begin{pr1}\label{t+1,t,1}
		Let $\nu,m,t,\mf,\ms$ and $X$ be as in Assumption 1. If $\tau_t(\ms)=t$ and $t+1\le\dim X\le m$, then $|\mf|<f_0(\nu,m,t)$.
	\end{pr1}
	\begin{proof}
		We investigate $\mf$ in two cases.
		
		\noindent\textbf{Case 1. $\dim X=t+1$.}
		
		In this case, we have $|\ms|=1$. Let $S$ be the unique member of $\ms$. Observe that
		\begin{equation}\label{t+1,t,1,1}
			\mf=\mf_S\cup\left(\bigcup_{R\in\mm(t;S)}\mf_R\bs\mf_S\right).
		\end{equation}
		For each $R\in\mm(t;S)$, since $\tau_t(\mf)=t+1$, there exists $F\in\mf\bs\mf_S$ such that $\dim(F\cap R)\le t-1$.  Note that $\dim(F\cap S)=t$.
		We get $\dim(F\cap R)=t-1$.
		Let $H_1=F+R$.
		We have $\dim H_1=m+1$.
		It is routine to check that $S\subset H_1$ and $\dim(F'\cap H_1)\ge t+1$ for any $F'\in\mf_R\bs\mf_S$.
		Then
		\begin{equation}\label{t+1,t,1,2}
			\mf_R\bs\mf_S\subset\bigcup_{I\in\mm(t+1;H_1)\bs\{S\},\ R\subset I}\mf_I.
		\end{equation}
		
		For each $I\in\mm(t+1;H_1)\bs\{S\}$ with $R\subset I$,  there exists $F''\in\mf$ such that $\dim(F''\cap I)\le t-1$. Then from $R\subset I$ and $\dim(S\cap F'')\ge t$, we have $\dim(F''\cap I)=t-1$. By Lemma \ref{FS}, we have
		$$|\mf_I|\le{m-t+1\b1}N(\nu,m,t+2).$$
		Note that
		$$\left|\left\{I\in\mm(t+1;H_1): R\subset S, I\neq S\right\}\right|\le{m-t+1\b1}-1=q{m-t\b1},$$
		$$|\mf_S|\le N(\nu,m,t+1),\quad\left|\mm(t;S)\right|={t+1\b1}.$$
		It follows from (\ref{t+1,t,1,1}) and (\ref{t+1,t,1,2}) that
		$$|\mf|\le N(\nu,m,t+1)+q{m-t\b1}{t+1\b1}{m-t+1\b1}N(\nu,m,t+2).$$
		By (\ref{FFS}), (\ref{JS}) and $2\nu\ge 2m+2t+6$, we obtain
		\begin{equation*}
			\begin{aligned}
				&\dfrac{f_0(\nu,m,t)-|\mf|}{N(\nu,m,t+2)}\\
				\ge&q{m-t\b1}\cdot\dfrac{q^{2(\nu-t-1)}-1}{q^{m-t-1}-1}-q{m-t+1\b2}-q{m-t\b1}{t+1\b1}{m-t+1\b1}\\
				>&q^{2\nu-2t-1}-q^{2(m-t)+1}-q^{2m-t+3}\\
				>&0,
			\end{aligned}
		\end{equation*}
		as desired.

		\noindent{\bf Case 2. $t+2\le\dim X\le m$.}

		Write $k=\dim X$. Suppose that $T\in\mp_t$ is contained in each member of $\ms$. For $G\in\mf\bs\mf_T$, let $H_2=T+G$.
		For each $G'\in\mf_T$, if $G'\cap X=T$, since $\dim(G\cap G')\ge t$ and $T\not\subset G$, we have $\dim(G'\cap H_2)\ge t+1$.
		On the other hand, if $T\subsetneq G'\cap X$, then there exists a member of $\mm(t+1;X)$ containing $T$.
		Write
		$$\ma:=\left\{A\in\mm(t+1;H_2): T\subset A\not\subset X\right\},\quad\mb:=\left\{B\in\mm(t+1;X): T\subset B\right\}.$$
		We have
		\begin{equation}\label{t+1,t,2,1}
			\mf_T=\left(\bigcup_{A\in\ma}\mf_A\right)\cup\left(\bigcup_{B\in\mb}\mf_B\right).
		\end{equation}
		For each $A\in\ma$, by assumption, $A$ is not a $t$-cover of $\mf$.
		Then there exists $G''\in\mf\bs\mf_T$ such that $\dim(G''\cap A)\le t-1$.
		Together with $\dim(G''\cap T)=t-1$,
		we have $\dim(G''\cap A)=t-1$.
		By Lemma \ref{FS}, we have
		$$|\mf_A|\le{m-t+1\b1}N(\nu,m,t+2).$$
		By Lemma \ref{XJ}, we get $\dim(T\cap G)=t-1$, which implies that $\dim H_2=m+1$ and $X=T+(G\cap X)\subset H_2$. Then
		$$|\ma|\le{m-t+1\b 1}-{k-t\b1}=q^{k-t}{m-k+1\b1},\quad|\mb|\le{k-t\b1}.$$
		Together with $|\mf_B|\le N(\nu,m,t+1)$ for each $B\in\mb$, from (\ref{t+1,t,2,1}) we obtain
		\begin{equation}\label{S}
			|\mf_T|\le{k-t\b1}N(\nu,m,t+1)+q^{k-t}{m-k+1\b1}{m-t+1\b1}N(\nu,m,t+2).
		\end{equation}
		Write $\mc=\{C\in\mm(k-1;X): T\not\subset C\}$. Suppose $X\not\in\mp_k$, then $X$ has type $(k,1)$. Note that $|\mc|\le q+1$.
		Then
		$$|\mf\bs\mf_T|\le(q+1)N(\nu,m,k-1).$$
		If $X\in\mp_k$, we have $|\mc|={k\b1}-{k-t\b1}=q^{k-t}{t\b1}$. Then
		$$|\mf\bs\mf_T|\le q^{k-t}{t\b1}(N(\nu,m,k-1)-N(\nu,m,k)).$$
		Since $2\nu\ge 3m+2$ and $m\ge k\ge t+2\ge3$, by (\ref{FFS}) and (\ref{JS}), we have
		\begin{equation*}
			\begin{aligned}
				&\dfrac{q^{k-t}{t\b1}(N(\nu,m,k-1)-N(\nu,m,k))}{N(\nu,m,k-1)}\\
				=&q^{k-t}{t\b1}\left(1-\dfrac{q^{m-k+1}-1}{q^{2(\nu-k+1)}-1}\right)\\
				>&q^{k-t}\cdot q^{t-1}\cdot\left(1-q^{m+k-2\nu-1}\right)\\
				\ge&q+1,
			\end{aligned}
		\end{equation*}
		which implies that
		\begin{equation}\label{C}|\mf\bs\mf_T|\le q^{k-t}{t\b1}(N(\nu,m,k-1)-N(\nu,m,k))\end{equation}
		whenever $X$ is a member of $\mp_k$ or not.
		
		Suppose $k\ge t+3$.
		For each $t+3\le l\le m$, write
		$$u(l)=q^{l-t}{t\b1}(N(\nu,m,l-1)-N(\nu,m,l)).$$
		For each $t+3\le l<m$, since $2\nu\ge3m+2>m+l$, by (\ref{FFS}), (\ref{JS}) and $u(l+1)\le q^{l-t+1}{t\b1}N(\nu,m,l)$, we have
		\begin{equation*}
			\begin{aligned}
				\dfrac{u(l)}{u(l+1)}&\ge\dfrac{1}{q}\cdot\left(\dfrac{q^{2(\nu-l+1)}-1}{q^{m-l+1}-1}-1\right)>q^{2\nu-m-l}-\dfrac{1}{q}>1.
			\end{aligned}
		\end{equation*}
		Hence
		\begin{equation*}
			\begin{aligned}
				&\dfrac{f_0(\nu,m,t)-|\mf|}{N(\nu,m,t+2)}\\
				\ge&q^{k-t}{m-k+1\b1}\cdot\dfrac{q^{2(\nu-t-1)}-1}{q^{m-t-1}-1}-q{m-t+1\b2}-q^{k-t}{m-k+1\b1}{m-t+1\b1}-u(t+3)\\
				>&q^{2\nu-2t-1}-q^{2(m-t)+1}-q^{2m-2t+2}-q^{t+3}\\
				\ge&0
			\end{aligned}
		\end{equation*}
		follows from (\ref{FFS}), (\ref{JS}), (\ref{S}), (\ref{C}) and $2\nu\ge\max\{3m+2,2m+2t+6\}$, as desired.

		Now suppose that $k=t+2$. For each $C\in\mc$, since $T\not\subset C$, $C$ is not a $t$-cover of $\mf$. Thus there exists $E\in\mf$ such that $\dim(E\cap C)\le t-1$. Notice that $\dim(E\cap X)\ge t$, which implies that $\dim(E\cap C)=t-1$.
		Then $|\mf_C|\le{m-t+1\b1}N(\nu,m,t+2)$ follows from Lemma \ref{FS}. Together with $|\mc|\le\max\left\{q+1,{t+2\b1}-{2\b1}\right\}=q^2{t\b1}$, we get
		\begin{equation}\label{D}
			|\mf\bs\mf_T|\le q^2{t\b1}{m-t+1\b1}N(\nu,m,t+2).
		\end{equation}
		Then from (\ref{FFS}), (\ref{JS}), (\ref{S}), (\ref{D}) and $2\nu\ge\max\{3m+2,2m+2t+6\}$, we obtain
		\begin{equation*}
			\begin{aligned}
				&\dfrac{f_0(\nu,m,t)-|\mf|}{N(\nu,m,t+2)}\\
				\ge& q^{k-t}{m-k+1\b1}\cdot\dfrac{q^{2(\nu-t-1)-1}}{q^{m-t-1}-1}-q{m-t+1\b2}-\left(q^{k-t}{m-k+1\b1}+q^2{t\b1}\right){m-t+1\b1}\\
				>&q^{2\nu-2t-1}-q^{2(m-t)+1}-q^{2m-2t+2}-q^{m+3}\\
				\ge&0,
			\end{aligned}
		\end{equation*}
		as desired.
	\end{proof}

	\begin{pr1}
		Let $\nu,m,t$ and $\mf,\ms$ be as in Assumption 1. If $\tau_t(\ms)=t+1$, then $\mf=\mh_2(Z)$ for some $Z\in\mp_{t+2}$.
	\end{pr1}

	\begin{proof}
		By $\tau_t(\ms)=t+1$ and Lemma \ref{XJ}, $\ms$ is a non-trivial $t$-intersecting subfamily of ${V\b t+1}$. Therefore, by \cite[Remarks (\ro2) in Section 9.3]{T1}, there exists $Z\in{V\b t+2}$ such that $\ms\subset{Z\b t+1}$.
		
		Pick different members $A,B$ of $\ms$. Let $E=A\cap B$. Since $\ms$ is non-trivial, there exists $C\in\ms$ such that $E\not\subset C$. For any $F\in\mf$, if $E\subset F$, together with $\dim(F\cap C)\ge t$, we get $\dim(F\cap Z)\ge t+1$. If $E\not\subset F$, then $F\cap A$ and $F\cap B$ are two different subspaces of $Z$ with dimension at least $t$, which implies that $\dim(F\cap Z)\ge t+1$. Therefore, $\mf\subset\mh_2(Z)$. By the maximality of $\mf$ and $\mh_2(Z)$ is $t$-intersecting, we have $\mf=\mh_2(Z)$.
		
		Observe that $A+B=Z$ and $\dim(A\cap B)=t$. If $Z\not\in\mp_{t+2}$, $Z$ has type $(t+2,1)$. Then $|\mf|=(q+1)N(\nu,m,t+1)$.
		To finish the proof, it is sufficient to show that
		$$(q+1)N(\nu,m,t+1)<h_2(\nu,m,t).$$
		Since $2\nu\ge3m+2$, by (\ref{FFS}) and (\ref{JS}), we have
		\begin{equation*}
			\begin{aligned}
				\dfrac{h_2(\nu,m,t)}{N(\nu,m,t+1)}&={t+2\b1}-q{t+1\b1}\cdot\dfrac{q^{m-t-1}-1}{q^{2(\nu-t-1)}-1}>q^{t+1}-q^{-3}>q+1,
			\end{aligned}
		\end{equation*}
		as desired.
	\end{proof}

	\subsection{The case $\tau_t(\mf)\ge t+2$}

	\begin{pr1}\label{t+2}
		Let $\nu,m,t$ be positive integers with $2\nu\ge\max\{3m+2,2m+2t+6\}$ and $m\ge t+2$. Suppose $\mf$ is a maximum sized non-trivial $t$-intersecting subfamily of $\mp_m$ with $\tau_t(\mf)\ge t+2$. Then $|\mf|<f_0(\nu,m,t)$.
	\end{pr1}
	\begin{proof}
		Assume that $\tau_t(\mf)=r$.
		For each $S\in\ms$, observe that
		$$\mf\subset\bigcup_{W\in\mm(t;S)}\mf_W.$$
		For each $W\in\mm(t;S)$, w.l.o.g., assume that $\mf_W\neq\emptyset$. It is routine to check that ${m-a\b t-a}\le{m-t+1\b1}^{t-a}$ for each $0\le a\le t-1$. Then by Lemma \ref{FS}, there exist some totally isotropic subspaces $W_1,\dots,W_u$ such that $W=W_0\subset W_1\subset\cdots\subset W_u$ with $\dim W_u\ge r$, $\dim W_{u-1}<r$ and $$|\mf_{W_{i-1}}|\le{m-t+1\b1}^{\dim W_i-\dim W_{i-1}}|\mf_{W_i}|$$
		for each $i=1,\dots,u$. Notice that $\dim W_u\le m$ from $$|\mf_{W_u}|\ge|\mf_W|/{k-t+1\b1}^{\dim W_u-t}>0.$$
		When $t+2\le a\le m-1$ and $2\nu\ge3m-t$, we have
		\begin{equation*}
			\dfrac{N(\nu,m,a)}{N(\nu,m,a+1)}=\dfrac{q^{2(\nu-a)-1}}{q^{m-a}-1}\ge q^{2\nu-m-a}\ge q^{m-t+1}\ge{m-t+1\b1}.
		\end{equation*}
		Consequently,
		$$|\mf_W|\le{m-t+1\b1}^{\dim W_u-t}N(\nu,m,\dim W_u)\le{m-t+1\b1}^{r-t}N(\nu,m,r).$$
		Together with $|\mm(t;S)|\le{r\b t}$, we get
		$$|\mf|\le {r\b t}{m-t+1\b1}^{r-t}N(\nu,m,r).$$
		For each integer $s$ with $t\le s\le m$, write
		$$g(s)={s\b t}{m-t+1\b1}^{s-t}N(\nu,m,s).$$
		For $t\le s\le m-1$, from (\ref{FFS}), (\ref{JS}) and $2\nu\ge3m+2$, we obtain
		$$\dfrac{g(s+1)}{g(s)}=\dfrac{(q^{s+1}-1)(q^{m-t+1}-1)(q^{m-s}-1)}{(q^{s-t+1}-1)(q-1)(q^{2(\nu-s)}-1)}<q^{2m+s+2-2\nu}\le q^{3m+1-2\nu}<1.$$
		Then
		$$|\mf|\le{t+2\b2}{m-t+1\b1}^2N(\nu,m,t+2).$$
		Since $2\nu\ge 2m+2t+6$, by (\ref{FFS}) and (\ref{JS}), we get
		\begin{equation*}
			\begin{aligned}
				&\dfrac{f_0(\nu,m,t)-|\mf|}{N(\nu,m,t+2)}\\
				\ge&{m-t+1\b1}\cdot\dfrac{q^{2(\nu-t-1)-1}}{q^{m-t-1}-1}-q{m-t+1\b2}-{t+2\b2}{m-t+1\b1}^2\\
				>&q^{2\nu-2t-1}-q^{2(m-t)+1}-q^{2(m+2)}\\
				\ge&0,
			\end{aligned}
		\end{equation*}
		as desired.
	\end{proof}

	\vskip0.1in
	\noindent\textsc{Acknowledgement.} This research is supported by NSFC (12071039).

\end{document}